\documentclass[10pt]{amsart}
\usepackage{amssymb}
\usepackage{amsmath,amssymb,amsfonts,amsthm,
latexsym, amscd, amsfonts, epsfig, epic}
\usepackage{mathrsfs}
\usepackage{color}
\usepackage{eucal}
\usepackage{eufrak}
\usepackage{xspace}
\usepackage{a4wide}
\usepackage{colordvi}
\usepackage{eucal}
\usepackage{centernot}
\newcounter{mycount}
\theoremstyle{plain}
\newtheorem{theorem}[mycount]{Theorem}

\newtheorem{lemma}[mycount]{Lemma}

\theoremstyle{definition}
\newtheorem{definition}{Definition}

\theoremstyle{example}

\theoremstyle{remark}

\numberwithin{equation}{section}
\numberwithin{definition}{section}
\numberwithin{remark}{section}
\numberwithin{figure}{section}
\begin{document}

\title{Heaps and Two Exponential Structures}

\author{Emma Yu Jin$^{\dag}$}
\thanks{$^{\dag}$ Corresponding author email: jin@cs.uni-kl.de; yu.jin@tuwien.ac.at. The author is supported by the German Research Foundation DFG, JI
207/1-1 and the Austrian Research
Fund FWF, project SFB F50 Algorithmic and Enumerative Combinatorics.} \email{jin@cs.uni-kl.de;yu.jin@tuwien.ac.at}
\address{Institut f\"ur Diskrete Mathematik und Geometrie, TU Wien, Wiedner Hauptstr. 8--10, 1040 Vienna, Austria}

\maketitle

\begin{abstract}
Take ${\sf Q}=({\sf Q}_1,{\sf Q}_2,\ldots)$ to be an exponential structure and $M(n)$ to be the number of minimal elements of ${\sf Q}_n$ where $M(0)=1$.
Then a sequence of numbers $\{r_n({\sf Q}_n)\}_{n\ge 1}$ is defined by the
equation
\begin{eqnarray*}
\sum_{n\ge 1}r_n({\sf Q}_n)\frac{z^n}{n!\,M(n)}=-\log(\sum_{n\ge
0}(-1)^n\frac{z^n}{n!\,M(n)}).
\end{eqnarray*}
Let $\bar{{\sf Q}}_n$ denote the poset ${\sf Q}_n$ with a $\hat{0}$ adjoined and let $\hat{1}$ denote the unique maximal element in the poset ${\sf Q}_n$. Furthermore, let $\mu_{{\sf Q}_n}$ be the M\"{o}bius function on the poset $\bar{{\sf Q}}_n$. Stanley proved that $r_n({\sf Q}_n)=(-1)^n\mu_{{\sf Q}_n}(\hat{0},\hat{1})$. This implies that the numbers $r_n({\sf Q}_n)$ are integers. In this paper, we study
the cases ${\sf Q}_n=\Pi_n^{(r)}$ and ${\sf Q}_n={\sf Q}_n^{(r)}$ where $\Pi_n^{(r)}$ and ${\sf Q}_n^{(r)}$ are posets, respectively, of set partitions of $[rn]$ whose block sizes are divisible by $r$ and of $r$-partitions of $[n]$. In both cases we prove that $r_n(\Pi_n^{(r)})$ and $r_n({\sf Q}_n^{(r)})$ enumerate the pyramids by applying the Cartier-Foata monoid identity and further prove that $r_n(\Pi_n^{(r)})$ is the generalized Euler number $E_{rn-1}$ and that $r_n({\sf Q}_n^{(2)})$ is the number of complete non-ambiguous trees of size $2n-1$ by bijections. This gives a new proof of Welker's theorem that
$r_n(\Pi_n^{(r)})=E_{rn-1}$ and implies the construction of
$r$-dimensional complete non-ambiguous trees. As a bonus of applying
the theory of heaps, we establish a bijection between the set of complete
non-ambiguous forests and the set of pairs of permutations with no
common rise. This answers an open question raised by Aval {\it et
al.}.
\end{abstract}
\section{Introduction}
We denote by $\Pi_n$ the poset of all the set partitions of $[n]$
ordered by refinement, that is, define $\sigma\le\pi$ if every block
of $\sigma$ is contained in a block of $\pi$. Let $\rho\in\Pi_n$ be
the minimal element of $\Pi_n$, i.e.,
$\rho=\{\{1\},\{2\},\ldots,\{n\}\}$. Consider an interval
$[\sigma,\pi]$ in the poset $\Pi_n$ and suppose
$\pi=\{B_1,B_2,\ldots,B_k\}$ and $B_i$ is partitioned into
$\lambda_i$ blocks in $\sigma$. Then we have
$[\sigma,\pi]\cong\Pi_{\lambda_1}\times\Pi_{\lambda_2}\times\cdots\times
\Pi_{\lambda_k}$. For the particular case $\sigma =\rho$, we have
$[\rho,\pi]\cong\Pi_{\vert B_1\vert}\times\Pi_{\vert
B_2\vert}\times\cdots\times \Pi_{\vert B_k\vert}$. If we set
$a_j=\vert\{i:\lambda_i=j\}\vert$ for every $j$, then we can rewrite
\begin{eqnarray}\label{E:sigmapi}
[\sigma,\pi]\cong \Pi_1^{a_1}\times\Pi_2^{a_2}\times\cdots\times\Pi_n^{a_n}.
\end{eqnarray}
The poset $\Pi_n$ of set partitions is the archetype of exponential
structures. The concept of exponential structure was introduced by
Stanley as a generalization of compositional and exponential
formulas \cite{Stanley:ex1,Stanley:ec2,Ehrenborg:09}. An {\em
exponential structure} is a sequence ${\sf Q}=({\sf Q}_1,{\sf Q}_2,\ldots)$ of
posets such that:
\begin{enumerate}
\item for each $n\in\mathbb{N}^+$, the poset ${\sf Q}_n$ is finite, has a unique
maximal element $\hat{1}$ and every maximal chain of ${\sf Q}_n$ has $n$
elements.
\item for $\pi\in {\sf Q}_n$, the interval $[\pi,\hat{1}]$ is
isomorphic to the poset $\Pi_k$ of set partitions for some $k$.
\item the subposet
$\Lambda_\pi=\{\sigma\in {\sf Q}_n:\sigma\le \pi\}$ of ${\sf Q}_n$ is isomorphic
to ${\sf Q}_1^{a_1}\times {\sf Q}_2^{a_2}\times\cdots\times {\sf Q}_n^{a_n}$ for
unique $a_1,a_2,\ldots,a_n\in\mathbb{N}$.
\end{enumerate}
Suppose $\pi\in {\sf Q}_n$ and $\rho$ is a minimal element of ${\sf Q}_n$
satisfying $\rho\le\pi$. By $(1)$ and $(2)$, we obtain that $[\rho,\hat{1}]\cong\Pi_n$. It follows from (\ref{E:sigmapi})
that $[\rho,\pi]\cong \Pi_1^{a_1}\times
\Pi_2^{a_2}\times\cdots\times \Pi_n^{a_n}$ for unique
$a_1,a_2,\ldots,a_n\in\mathbb{N}$ satisfying $\sum_{i}ia_i=n$ and
$\sum_{i}a_i=\vert\pi\vert$. In particular, if $\rho_1$ is another
minimal element of ${\sf Q}_n$ satisfying $\rho_1\le \pi$, then we have
$[\rho_1,\pi]\cong [\rho,\pi]$.

We will define the numbers $r_n({\sf Q}_n)$ associated with an exponential
structure ${\sf Q}=({\sf Q}_1,{\sf Q}_2,\ldots)$ in the following way. Let
$M(n)$ be the number of minimal elements of ${\sf Q}_n$ for $n\ge 1$ and set $M(0)=1$. Then a sequence of numbers $\{r_n({\sf Q}_n)\}_{n\ge 1}$ is
defined by the equation
\begin{eqnarray}\label{E:defr}
\sum_{n\ge 1}r_n({\sf Q}_n)\frac{z^n}{n!\,M(n)}=-\log(\sum_{n\ge
0}(-1)^n\frac{z^n}{n!\,M(n)}).
\end{eqnarray}
Let $\bar{{\sf Q}}_n$ denote the poset ${\sf Q}_n$ with a $\hat{0}$ adjoined and let $\hat{1}$ denote the unique maximal element in the poset ${\sf Q}_n$. Furthermore, let $\mu_{{\sf Q}_n}$ be the M\"{o}bius function on the poset $\bar{{\sf Q}}_n$.
Then from Chapter 5.5 of \cite{Stanley:ec2}, we know
\begin{eqnarray}\label{E:Qexp}
\sum_{n\ge
1}\mu_{{\sf Q}_n}(\hat{0},\hat{1})\frac{z^n}{n!M(n)}=-\log(\sum_{n\ge
0}\frac{z^n}{n!M(n)}),
\end{eqnarray}
and thus, $r_n({\sf Q}_n)=(-1)^n\mu_{{\sf Q}_n}(\hat{0},\hat{1})$. This
implies that the numbers $r_n({\sf Q}_n)$ are integers for any exponential
structure ${\sf Q}=({\sf Q}_1,{\sf Q}_2,\ldots)$. In the case ${\sf Q}_n=\Pi_n$, the
number $M(n)$ of minimal elements in the poset $\Pi_n$ is $1$. It follows immediately from (\ref{E:Qexp}) and (\ref{E:defr}) that $r_n(\Pi_n)=\mu_{\Pi_n}(\hat{0},\hat{1})=0$ for $n\ge 2$ and $r_1(\Pi_1)=-\mu_{\Pi_1}(\hat{0},\hat{1})=1$. There are three other examples of exponential structures ${\sf Q}=({\sf Q}_1,{\sf Q}_2,\ldots)$ in \cite{Stanley:ex1,Stanley:ec2}.
\begin{enumerate}
\item ${\sf Q}_n={\sf Q}_n(q)$ which is the poset of direct sum decompositions
of the $n$-dimensional vector space $V_n(q)$ over the finite field
$\mathbb{F}_{q}$. Let $V_n(q)$ be an $n$-dimensional vector space over the finite field $\mathbb{F}_q$. Let ${\sf Q}_n(q)$ consist of all the collections $\{W_1,W_2,\ldots,W_k\}$ of subspaces of $V_n(q)$ such that $\dim W_i>0$ for all $i$, and such that $V_n(q)=W_1\oplus W_2\oplus \cdots \oplus W_k$ (direct sum). An element of ${\sf Q}_n(q)$ is called a {\em direct sum decomposition} of $V_n(q)$. We order ${\sf Q}_n(q)$ by refinement, i.e., $\{W_1,W_2,\ldots,W_k\}\le \{W_1',W_2',\ldots,W_j'\}$ if each $W_r$ is contained in some $W_s'$.
\item ${\sf Q}_n=\Pi_{n}^{(r)}$ which is the poset of set partitions of $[rn]$ whose block sizes are divisible by $r$.
\item ${\sf Q}_n={\sf Q}_n^{(r)}$ which is the poset of $r$-partitions of $[n]$. The definition of $r$-partition will be given in Section~\ref{S:rpar}.
\end{enumerate}
The poset ${\sf Q}_n(q)$ had been studied by Welker \cite{Welker} who used the theory of free monoid to give an expression of the M\"{o}bius
function $\mu_{{\sf Q}_n(q)}(\hat{0},\hat{1})$; see Theorem 4.4 in \cite{Welker}.

Here we prove that the number $r_n(\Pi_n^{(r)})$ is the generalized Euler
number $E_{rn-1}$ and $r_n({\sf Q}_n^{(2)})$ is the number of complete non-ambiguous trees of size $2n-1$.
First, we show that $r_n(\Pi_n^{(r)})$ and $r_n({\sf Q}_n^{(r)})$ enumerate the pyramids by applying the Cartier-Foata monoid identity on the respective posets $\Pi_n^{(r)}$
and ${\sf Q}_n^{(r)}$. Secondly, we build a bijection between the set of the
pyramids counted by $r_n(\Pi_n^{(r)})$ and the set of permutations
of $[rn-1]$ with descent set $\{r,2r,\ldots,rn-r\}$. This gives a new
proof of Welker's theorem that $r_n(\Pi_n^{(r)})=E_{rn-1}$. Welker
\cite{Welker} proved that $r_n(\Pi_n^{(r)})=\vert\mu_{\Pi_n^{(r)}}(\hat{0},\hat{1})\vert
=E_{rn-1}$ by counting the number of descending chains in the chain
lexicographic shellable poset $\Pi_n^{(r)}$. Thirdly, we establish
a bijection between the set of the pyramids counted by $r_n({\sf Q}_n^{(2)})$ and the
set of complete non-ambiguous trees of size $2n-1$. This implies the construction of $r$-dimensional complete non-ambiguous trees. As a bonus of applying the theory of heaps, we provide a bijection between the set of complete
non-ambiguous forests and the set of pairs of permutations with no
common rise. This answers an open question raised by Aval {\it et al.}.

This paper is organized as follows. In section~\ref{S:hmFoata}, we introduce the terms heap, pyramid and the Cartier-Foata identity. In
sections~\ref{S:parde} and \ref{S:rpar}, we prove our main results. That is, that
$r_n(\Pi_n^{(r)})$ is the generalized Euler numbers
$E_{rn-1}$ and $r_n({\sf Q}_n^{(2)})$ is the number of complete
non-ambiguous trees of size $2n-1$. In subsection~\ref{S:forAval}, we establish a
bijection between the set of complete non-ambiguous forests and the
set of pairs of permutations with no common rise.
\section{Heap, Monoid and the Cartier-Foata identity}\label{S:hmFoata}
The theory of heaps was introduced by Viennot to interpret the
elements of the Cartier-Foata monoid in a geometric manner; see
\cite{Viennot,Kratthen,Vonline}. Here we will adopt the notations
from \cite{Viennot,Kratthen}. Let $\CMcal{B}$ be a set of pieces
with a symmetric and reflexive binary relation $\CMcal{R}$, i.e., we have
$a\CMcal{R}b\Leftrightarrow b\CMcal{R}a$ and $a\CMcal{R}a$ for every
$a,b\in \CMcal{B}$. We will define a {\em heap} as follows:
\begin{definition}\label{D:heap}
A {\em heap} is a triple $H=(P,\le,\varepsilon)$ where
$(P,\le)$ is a finite poset and $\varepsilon$ is a map $\varepsilon:P\mapsto \CMcal{B}$ satisfying:
\begin{enumerate}
\item for every $x,y\in P$, if $\varepsilon(x)\CMcal{R}\varepsilon(y)$,
then $x$ and $y$ are comparable, i.e., $x\le y$ or $y\le x$.
\item for every $x,y\in P$, if $y$ covers $x$
(i.e., $x\le y$ and for any $z$ such that $x\le z\le y$, $z=x$ or
$z=y$), then $\varepsilon(x)\CMcal{R}\varepsilon(y)$.
\end{enumerate}
\end{definition}
We represent the heap $H=(P,\le,\varepsilon)$ by the Hasse diagram of the poset $(P,\le)$ where every element $x$ is labelled by $\varepsilon(x)$, and we will give an example of heap from \cite{Viennot}. \\[7pt]
{\bf Example $1$.} Let $\CMcal{B}_1=\mathbb{Z}$ be a set of integers with a binary relation $\CMcal{R}_1$ defined by: $i\CMcal{R}_1 j$ if and only if $\vert i-j\vert\le 1$ for every $i,j\in \CMcal{B}_1$. The poset $(P_1,\le)$ is defined by its Hasse diagram in Figure~\ref{F:1} and the map $\varepsilon: P_1\mapsto \CMcal{B}_1$ is defined as such: if $\alpha\in P_1$ lies on the vertical line $x=i$, then $\varepsilon(\alpha)=i$. See Figure~\ref{F:1}. It is clear that $(P_1,\le,\varepsilon)$ satisfies conditions $(1)$ and $(2)$ in Definition~\ref{D:heap}.\\[7pt]
{\bf Remark $1$.} A heap $H=(P,\le,\varepsilon)$ is equivalent to the directed graph given as we will now state. Let $G(P)$ be the graph endowed with an orientation $\gamma$ whose vertex set is $P$ and such that there is an oriented edge from $x$ to $y$ if and only if $y<x$ in the poset $P$. Since $\le$ is an order relation in the poset $(P,\le)$, $\gamma$ is acyclic. Then the heap $(P,\le,\varepsilon)$ is equivalent to the directed graph $(G(P),\gamma)$.
\begin{figure}[htbp]
\begin{center}
\includegraphics[scale=0.5]{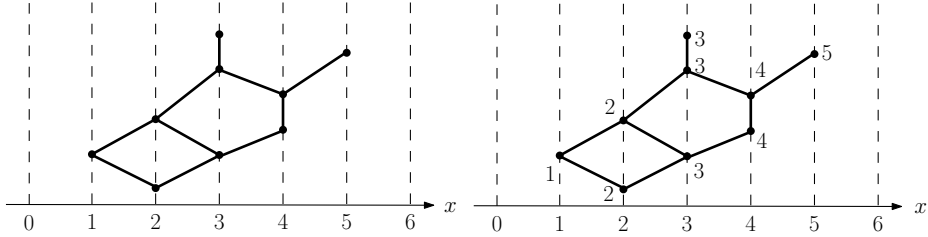}
\caption{The poset $(P_1,\le)$ (left) and the heap $H_1=(P_1,\le,\varepsilon)$ of integers (right).
\label{F:1}}
\end{center}
\end{figure}

We denote $H=\varnothing$ if the heap $H$ is empty. To clarify our subsequent discussion on the set of heaps,
we need to first review some definitions.
\begin{definition}\label{D:monoid}
A {\em monoid} $S$ is a set
that is closed under an associative binary operation $\circ$ and has
an identity element ${\rm id}\in S$ such that for all $s\in S$, we have ${\rm id} \circ s=s\circ{\rm id}=s$.
\end{definition}
\begin{definition}\label{D:tranclosure}
The {\em transitive closure} of a binary relation $\CMcal{R}$ on a set $X$ is the transitive relation $\CMcal{R}^*$ on the set $X$ such that $\CMcal{R^*}$ is minimal and contains $\CMcal{R}$.
\end{definition}
Let $\CMcal{H}(\CMcal{B},\CMcal{R})$ be the set of heaps $H=(P,\le,\varepsilon)$ given in Definition \ref{D:heap}, in particular, include $\varnothing\in \CMcal{H}(\CMcal{B},\CMcal{R})$. Then the set $\CMcal{H}(\CMcal{B},\CMcal{R})$ of heaps can be equipped with a monoid structure as follows:
\begin{definition}\label{D:circ}
Suppose $H_1=(P_1,\le_1,\varepsilon_1)$,
$H_2=(P_2,\le_2,\varepsilon_2)\in \CMcal{H}(\CMcal{B},\CMcal{R})$, and the composition $H_1\circ H_2=(P_3,\le_3,\varepsilon_3)$ is the heap defined by the following:
\begin{enumerate}
\item
$P_3$ is the disjoint union of $P_1$ and $P_2$.
\item
$\varepsilon_3$ is the unique map $\varepsilon_3: P_3\mapsto \CMcal{B}$ such that $\varepsilon_3=\varepsilon_i$ if we restrict $P_3$ to $P_i$ for $i=1,2$.
\item
The partial order $\varepsilon_3$ is the transitive closure of the following
relation $\CMcal{R}^*$. For $x,y\in P_3$, $x\CMcal{R}^* y$ if and only if one of $(a),(b),(c)$ is satisfied.
\begin{enumerate}
\item $x\le_1 y$ and $x,y\in P_1$
\item $x\le_2 y$ and $x,y\in P_2$
\item $x\in P_1,y\in P_2$ and
$\varepsilon_1(x)\CMcal{R}\varepsilon_2(y)$.
\end{enumerate}
\end{enumerate}
\end{definition}
Simply put, the heap $H_1\circ H_2$ is obtained by putting the pieces in $H_2$ on top of $H_1$ in its Hasse diagram representation. See Figure~\ref{F:2}.
\begin{figure}[htbp]
\begin{center}
\includegraphics[scale=0.5]{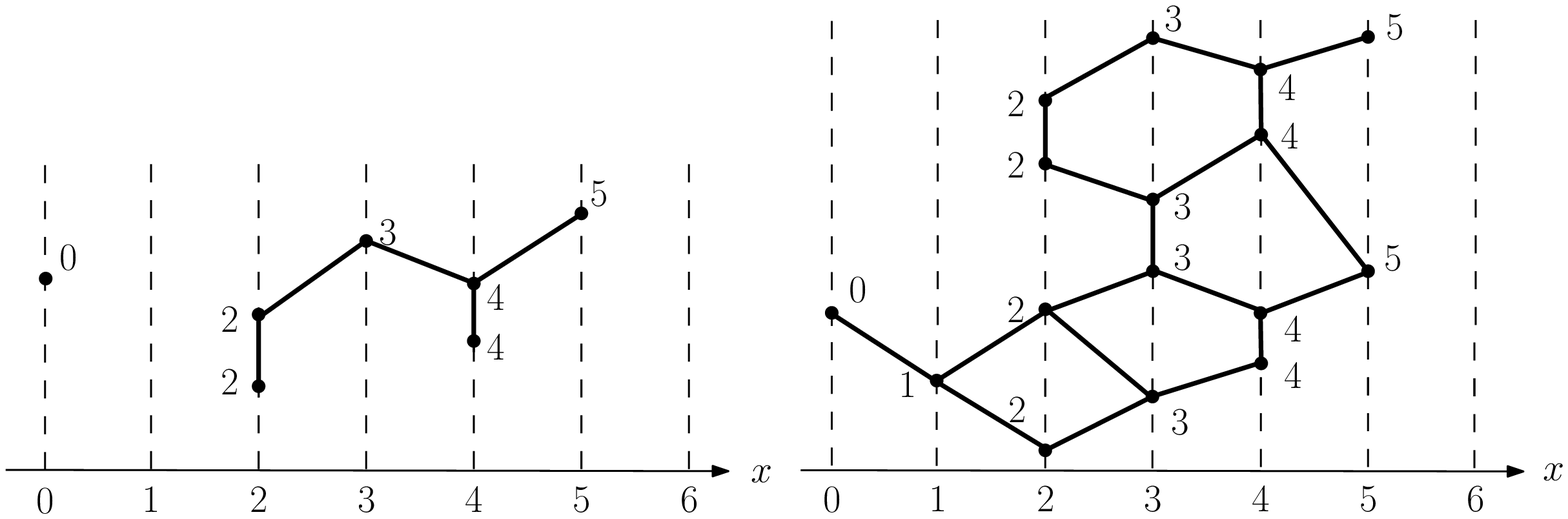}
\caption{The heap $H_2$ of integers (left) and the composition $H_1\circ H_2$ (right) where $H_1$ is given in Figure~\ref{F:1}.
\label{F:2}}
\end{center}
\end{figure}
By Definition~\ref{D:circ}, the set $\CMcal{H}(\CMcal{B},\CMcal{R})$ is closed under the associative composition $\circ$ and the identity element ${\rm id}$ is the empty heap $\varnothing$ from which it follows that the set $\CMcal{H}(\CMcal{B},\CMcal{R})$ is a monoid. The size of a heap $H$, denoted by $\vert H \vert$, is the number of pieces in the heap $H$.
\begin{definition}\label{D:triheap}
A {\em trivial heap} $T$ is a heap consisting of pieces
that are pairwise unrelated, i.e.,
$\varepsilon(x)\centernot{\CMcal{R}}\varepsilon(y)$ for every distinct $x,y$ in $T$.
\end{definition}
Let $\CMcal{T}(\CMcal{B},\CMcal{R})$ be the set of trivial heaps
contained in the set $\CMcal{H}(\CMcal{B},\CMcal{R})$. Let
$\mathbb{Z}[[\CMcal{H}(\CMcal{B},\CMcal{R})]]$ be the ring of
formal power series in $\CMcal{H}(\CMcal{B},\CMcal{R})$ with
coefficients in the commutative ring $\mathbb{Z}$. Then the
Cartier-Foata identity in the ring
$\mathbb{Z}[[\CMcal{H}(\CMcal{B},\CMcal{R})]]$ is
\begin{equation}\label{E:CF1}
\sum_{\substack{H\in \CMcal{H}(\CMcal{B},\CMcal{R})\\
T\in \CMcal{T}(\CMcal{B},\CMcal{R})}}(-1)^{\vert T\vert}\cdot (\,T\circ
H)=\varnothing.
\end{equation}
\begin{definition}\label{D:CFmonoid}
Consider the set of letters $S=\{x_a:a\in\CMcal{B}\}$ and the set
$${\sf R}=\{x_ax_b=x_bx_a,\mbox{ for }\, a,b\in\CMcal{B}\,\mbox{ such that } a\centernot{\CMcal{R}}b\}$$ of relations. Then the {\em Cartier-Foata monoid} is the set $\CMcal{M}(\CMcal{B},\CMcal{R})$ of equivalent classes of words of $S$ subject to the relations in ${\sf R}$.
\end{definition}
In other words, two words are equivalent if we can transform one into another by applying successive relations of ${\sf R}$. The Cartier-Foata monoid
$\CMcal{M}(\CMcal{B},\CMcal{R})$ is isomorphic to the monoid
$\CMcal{H}(\CMcal{B},\CMcal{R})$. Here we only show the isomorphism
$\varphi:\CMcal{H}(\CMcal{B},\CMcal{R})\rightarrow
\CMcal{M}(\CMcal{B},\CMcal{R})$ without the proof. We refer the readers to the Proposition 3.4 in \cite{Viennot} for a detailed proof that $\varphi$ is an isomorphism.

Given a heap $H=(P,\le,\varepsilon)\in \CMcal{H}(\CMcal{B},\CMcal{R})$ of size $n$, let $\sigma:P\mapsto [n]$ be a bijection such that $\sigma(a)\le \sigma(b)$ for every $a,b\in P$ where $a\le b$, and define a word
$w_{\sigma}=x_{m_1}x_{m_2}\cdots x_{m_n}$ where $m_i=\varepsilon(\sigma^{-1}(i))\in\CMcal{B}$. Here we call the bijection $\sigma$ a {\em natural labeling} (sometimes also called a {\em linear extension}) of the poset $(P,\le)$. For
any two different natural labelings $\sigma,\theta$ of the poset
$(P,\le)$, it was proven in \cite{Viennot} that $w_\sigma=w_\theta$ in the monoid $\CMcal{M}(\CMcal{B},\CMcal{R})$. In view of this, we set $\varphi(H)=w_{\sigma}$ and, thus, $\varphi(H)$ is well-defined.

In particular, we have $\varphi(\varnothing)=1$ and, for any two heaps
$H_1,H_2\in \CMcal{H}(\CMcal{B},\CMcal{R})$, we have
$\varphi(H_1\circ H_2)=\varphi(H_1)\cdot \varphi(H_2)$. Namely, under the isomorphism $\varphi$, the empty heap $\varnothing$ corresponds to
the identity word $1$ and the composition of two heaps $H_1\circ H_2$
corresponds to the product of two words $\varphi(H_1)$ and $\varphi(H_2)$ by
juxtaposition.\\[7pt]
{\bf Example $2$.} We continue using the heap $H_1=(P_1,\le,\varepsilon)$ in Example $1$ to show the isomorphism $\varphi$. We start labeling the minimal elements of $(P_1,\le)$ with integers $1,2,\ldots$ from left to right, then we remove these minimal elements from the poset $(P_1,\le)$ and label the minimal elements of the remaining poset by increasing integers from left to right.
We continue this process until all the elements in the poset $(P_1,\le)$ are labeled. It is clear this labeling process is a natural labeling of the poset $(P_1,\le)$ which gives us the corresponding word $\varphi(H_1)=x_2x_1x_3x_2x_4x_4x_3x_5x_3$ in the monoid $\CMcal{M}(\CMcal{B}_1,\CMcal{R}_1)$. Similarly, we get the corresponding words for the heaps $H_2$ and $H_1\circ H_2$ in Figure~\ref{F:2}, which are $\varphi(H_2)=x_0x_2x_4x_2x_4x_3x_5$ and $\varphi(H_1\circ H_2)=x_2x_1x_3x_0x_2x_4x_4x_3x_5x_3x_2x_4x_2x_4x_3x_5$. It is easy to verify that
$\varphi(H_1\circ H_2)=\varphi(H_1)\cdot\varphi(H_2)$ because $x_0$ commutes with every $x_i$ for $i\ge 2$.\\[7pt]
Let $\mathbb{Z}[[\CMcal{M}(\CMcal{B},\CMcal{R})]]$ be the ring of formal power series in $\CMcal{M}(\CMcal{B},\CMcal{R})$ with
coefficients in the commutative ring $\mathbb{Z}$. Then we can express ~(\ref{E:CF1}) in the ring
$\mathbb{Z}[[\CMcal{M}(\CMcal{B},\CMcal{R})]]$ as follows:
\begin{equation}\label{E:CF2}
\sum_{\substack{H\in \CMcal{H}(\CMcal{B},\CMcal{R})\\
T\in \CMcal{T}(\CMcal{B},\CMcal{R})}}(-1)^{\vert T\vert}
\cdot(\varphi(T)\cdot \varphi(H))=1.
\end{equation}
\begin{definition}
A {\em pyramid} is a heap with exactly one maximal element.
\end{definition}
Let $\CMcal{P}(\CMcal{B},\CMcal{R})$ denote the set of pyramids consisting of pieces in $\CMcal{B}$. Then, according to the exponential formula for the unlabeled combinatorial objects \cite{Stanley:ec2},
\begin{equation}\label{E:exph}
\sum_{H\in
\CMcal{H}(\CMcal{B},\CMcal{R})}\varphi(H)=_{\mbox{\scriptsize{comm}}}
\exp(\sum_{P\in
\CMcal{P}(\CMcal{B},\CMcal{R})}\frac{\varphi(P)}{\vert P\vert}).
\end{equation}
where $=_{\mbox{\scriptsize{comm}}}$ means the identity holds in the commutative extension of $\CMcal{H}(\CMcal{B},\CMcal{R})$, that is, in the commutative monoid which arises from $\CMcal{H}(\CMcal{B},\CMcal{R})$ by letting all pieces in $\CMcal{B}$ commute. In combination of (\ref{E:CF2}), it follows that
\begin{equation}\label{E:logpy}
\sum_{P\in \CMcal{P}(\CMcal{B},\CMcal{R})}\frac{\varphi(P)}{\vert P\vert} =_{\mbox{\scriptsize{comm}}} -\log(\sum_{T\in
\CMcal{T}(\CMcal{B},\CMcal{R})}(-1)^{\vert T\vert}\varphi(T)).
\end{equation}
We will mainly use (\ref{E:logpy}) to connect the numbers $r_n(\Pi_n^{(r)})$ and $r_n({\sf Q}_n^{(r)})$ with the pyramids in Section~\ref{S:parde} and Section~\ref{S:rpar}. It is worthwhile to mention that via (\ref{E:logpy}) Josuat-Verg\`{e}s had proved that a new sequence enumerates the pyramids which are related to the perfect matchings in \cite{Josuat:12}. With a slight abuse of notation, we simply use the labels $\varepsilon(x)$ in the heap $H=(P,\le,\varepsilon)$ to represent the elements $x\in P$ in the heap and we use $H=(P,\le)$ to denote the heap $H=(P,\le,\varepsilon)$.

\section{On the poset $\Pi_n^{(r)}$}\label{S:parde}
The poset of the set partitions of $[rn]$ whose block sizes are divisible by $r$ is $\Pi_n^{(r)}$ and its elements are ordered by refinement. In particular, $\Pi_n^{(1)}=\Pi_n$. Let $\Pi_{n,r}$ be the set of all the set partitions of $[rn]$ whose block sizes
are exactly $r$. Then $\Pi_{n,r}$ is the set of minimal elements of the poset $\Pi_n^{(r)}$ and, consequently, $M(n)=\vert \Pi_{n,r}\vert=(rn)!(n!r!^n)^{-1}$. For $r\ge 2$, the sequence $\{r_n(\Pi_n^{(r)})\}_{n\ge 1}$ defined by (\ref{E:defr}) satisfies
\begin{eqnarray}\label{E:rQ2}
\sum_{n\ge
1}\frac{r_n(\Pi_n^{(r)})z^{rn}}{(rn)!}=-\log(\sum_{n\ge
0}(-1)^n\frac{z^{rn}}{(rn)!}).
\end{eqnarray}
Let $\mathfrak{S}_n$ be the set of permutations of $[n]$. For $\pi\in\mathfrak{S}_n$, we write $\pi=a_1a_2\ldots a_n$ if $\pi(j)=a_j$. Then the descent set of permutation $\pi=a_1a_2\ldots a_n\in\mathfrak{S}_n$ is $\mbox{Des}\,\pi=\{i: 1\le i<n\, \mbox{ and } a_i>a_{i+1}\}$. The generalized Euler number $E_{rn-1}$ counts the number of permutations $\pi\in\mathfrak{S}_{rn-1}$ such that $\mbox{Des}\,\pi=\{r,2r,\ldots,rn-r\}$.
One of our main results is a new proof
of the following result due to Welker \cite{Welker}:
\begin{theorem}\label{T:Eulercm}
$r_n(\Pi_n^{(r)})=E_{rn-1}$.
\end{theorem}
We will prove Theorem~\ref{T:Eulercm} by first showing that $r_n(\Pi_n^{(r)})$ enumerates the pyramids of the monoid $\CMcal{H}(\CMcal{B}_2,\CMcal{R}_2)$ in
Lemma~\ref{L:cirn} and then building a bijection between the set of the pyramids counted by $r_n(\Pi_n^{(r)})$ and the set of permutations of $[rn-1]$ with descent set $\{r,2r,\ldots,rn-r\}$ in Lemma~\ref{L:bijrn}.

Before we proceed, we introduce some definitions and notations. We say $\pi=a_1a_2\cdots a_{2n-1}$ is an {\em alternating permutation} if
$a_1<a_2>\cdots<a_{2n-2}>a_{2n-1}$, namely, Des
$\pi=\{2,4,\ldots,2n-2\}$. It is well-known that the number of alternating permutations of $[2n-1]$ is counted by the Euler number $E_{2n-1}$ (also called the tangent number), whose exponential generating function is
\begin{eqnarray}\label{E:Euler2}
\sum_{n\ge 1}E_{2n-1}\frac{z^{2n-1}}{(2n-1)!}=\tan(z).
\end{eqnarray}
\subsection{Connection between $r_n(\Pi_n^{(r)})$ and heaps}
Recall that $\Pi_{n,r}$ is the set of set partitions of $[rn]$ where
each block has size $r$. We say $\{i_1,\ldots,i_r\}$ is a block of
partition $\pi\in \Pi_{n,r}$, denoted by $\{i_1,\ldots,i_r\}\in
\pi$, if $1\le i_1<\cdots<i_r\le rn$. The diagram
representation of $\pi\in \Pi_{n,r}$ is given as follows: We draw
$rn$ dots in a line labeled with $1,2,\ldots,rn$. For each block
$\{i_1,\ldots,i_r\}\in\pi$, we connect $i_j$ and $i_{j+1}$ by an
edge for all $1\le j<r$. We say two blocks $\{i_1,\ldots,i_r\}$ and
$\{j_1,\ldots,j_r\}$ of $\pi$ are not crossing if $i_r<j_1$ or
$j_r<i_1$. Otherwise two blocks $\{i_1,\ldots,i_r\}$ and
$\{j_1,\ldots,j_r\}$ are crossing. See Figure~\ref{F:41} for an example.
\begin{figure}[htbp]
\begin{center}
\includegraphics[scale=0.8]{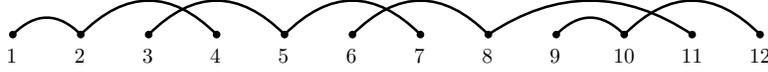}
\caption{The diagram representation of $\pi$ where $\pi\in\Pi_{4,3}$ and $\pi$ has blocks $\{1,2,4\},\{3,5,7\},\{6,8,11\},\{9,10,12\}$.
\label{F:41}}
\end{center}
\end{figure}

Stanley \cite{Stanley:ec2} proved $r_n(\Pi_n^{(2)})=E_{2n-1}$ by
using the exponential generating function of $E_{2n-1}$ given
in (\ref{E:Euler2}). More precisely, the hyperbolic tangent
function $\tanh(z)$ given by
$\tanh(z)=(e^z-e^{-z})(e^z+e^{-z})^{-1}$ satisfies
$\tanh(z)=-i\tan(iz)$. By setting $r=2$ to (\ref{E:rQ2}) and replacing
$z^2$ by $-z^2$, we get
\begin{eqnarray*}
\sum_{n\ge
1}r_n(\Pi_n^{(2)})(-1)^n\frac{z^{2n}}{(2n)!}=-\log(\sum_{n\ge
0}\frac{z^{2n}}{(2n)!}).
\end{eqnarray*}
By differentiating both sides of the above equation with respect to $z$, we have
\begin{eqnarray*}
\sum_{n\ge 1}(-1)^n \frac{r_n(\Pi_n^{(2)})z^{2n-1}}{(2n-1)!}
=-(\sum_{n\ge 0}\frac{z^{2n}}{(2n)!})^{-1}\sum_{n\ge 1}
\frac{z^{2n-1}}{(2n-1)!}=-\frac{e^{z}-e^{-z}}{e^{z}+e^{-z}}=-\tanh(z).
\end{eqnarray*}
In view of (\ref{E:Euler2}) and $\tanh(z)=-i\tan(iz)$, we obtain
\begin{eqnarray*}
\sum_{n\ge 1}(-1)^n r_n(\Pi_n^{(2)})\frac{z^{2n-1}}{(2n-1)!}
=-\tanh(z)=\sum_{n\ge 1}(-1)^n E_{2n-1}\frac{z^{2n-1}}{(2n-1)!},
\end{eqnarray*}
which yields $r_n(\Pi_n^{(2)})=E_{2n-1}$. For general $r$, we shall show that $r_n(\Pi_n^{(r)})$ counts the number of pyramids related to the set $\Pi_{n,r}$.

Let $\mathcal{B}_2=\{\{i_1,i_2,\ldots,i_r\}:1\le i_1<i_2<\cdots<i_r\}$ be the set of blocks with a binary relation $\CMcal{R}_2$ defined by $\{i_1,i_2,\ldots,i_r\}\CMcal{R}_2 \{j_1,j_2,\ldots,j_r\}$ if and
only if two blocks $\{i_1,i_2,\ldots,i_r\}$ and $\{j_1,j_2,\ldots,j_r\}$ are
crossing. Let $(P,\le)$ be a poset where every element is labeled by an element
$\{i_1,i_2,\ldots,i_r\}\in\CMcal{B}_2$ such that:
\begin{enumerate}
\item $\{i_1,i_2,\ldots,i_r\}$ and $\{j_1,j_2,\ldots,j_r\}$ are
comparable if they are crossing.
\item if $\{j_1,j_2,\ldots,j_r\}$ covers $\{i_1,i_2,\ldots,i_r\}$ in the
poset $(P,\le)$, then they are crossing.
\end{enumerate}
By definition~\ref{D:heap}, the poset $(P,\le)$ is a heap
$H=(P,\le)\in\CMcal{H}(\CMcal{B}_2,\CMcal{R}_2)$. In a geometric way, we
can represent each block $\{i_1,i_2,\ldots,i_r\}$ by a line
$i_1-i_2-\cdots-i_r$ and put $i_1-i_2-\cdots-i_r$ above $j_1-j_2-\cdots-j_r$ if
$\{j_1,j_2,\ldots,j_r\} \le \{i_1,i_2,\ldots,i_r\}$ in the heap
$H=(P,\le)$ and the line $i_1-i_2-\cdots-i_r$ cannot move downwards without touching the line $j_1-j_2-\cdots-j_r$. See Figure~\ref{F:3}.
\begin{figure}[htbp]
\begin{center}
\includegraphics[scale=0.7]{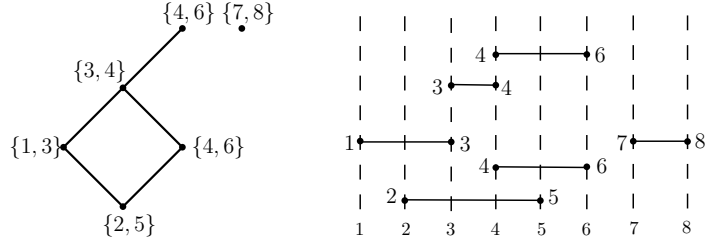}
\caption{A heap $H$ (left) in the monoid $\CMcal{H}(\CMcal{B}_2,\CMcal{R}_2)$ and its geometric representation (right). The corresponding word in the monoid $\CMcal{M}(\CMcal{B}_2,\CMcal{R}_2)$ is $\varphi(H)=x_{\{2,5\}}x_{\{7,8\}}x_{\{1,3\}}
x_{\{4,6\}}x_{\{3,4\}}x_{\{4,6\}}$.
\label{F:3}}
\end{center}
\end{figure}
Now we are in a position to establish a connection between $r_n(\Pi_n^{(r)})$ and the set of pyramids $\CMcal{P}(\CMcal{B}_2,\CMcal{R}_2)$ via (\ref{E:logpy}).
\begin{lemma}\label{L:cirn}
$r_n(\Pi_n^{(r)})$ counts the number of pyramids $P_n^*$ in $\CMcal{P}(\CMcal{B}_2,\CMcal{R}_2)$
such that $P_n^*$ has $n$ elements which are exactly the blocks of a partition from $\Pi_{n,r}$ and
the unique maximal element of $P_n^*$ is the block containing $rn$.
\end{lemma}
\begin{proof}
From (\ref{E:logpy}) we know
\begin{equation}\label{E:1}
\sum_{P\in \CMcal{P}(\CMcal{B}_2,\CMcal{R}_2)}\frac{\varphi(P)}{\vert P\vert} =_{\mbox{\scriptsize{comm}}} -\log(\sum_{T\in
\CMcal{T}(\CMcal{B}_2,\CMcal{R}_2)}(-1)^{\vert T\vert}\varphi(T)).
\end{equation}
We now consider the Cartier-Foata monoid $\CMcal{M}(\CMcal{B}_2,\CMcal{R}_2)$.
Let $\{a_i\}_{i\ge 1}$ be a sequence of variables such that $a_i^2=0$ for
every $i$ and $a_ia_j=a_ja_i$ for every $i,j$. Furthermore let
$f:\mathbb{Z}[[\CMcal{M}(\CMcal{B}_2,\CMcal{R}_2)]]\rightarrow
\mathbb{Z}[[a_1,a_2,\ldots]]$
be a ring homomorphism such that
$f(x_{\{i_1,i_2,\ldots,i_r\}})=\prod_{j=1}^r a_{i_j}$ and $f(1)=1$. In other words, the commutative ring $\mathbb{Z}[[a_1,a_2,\ldots]]$ is a commutative extension of the ring
$\mathbb{Z}[[\CMcal{M}(\CMcal{B}_2,\CMcal{R}_2)]]$. We apply $f$ on both sides of (\ref{E:1}) and obtain
\begin{equation}\label{E:logpy1}
\sum_{P\in \CMcal{P}(\CMcal{B}_2,\CMcal{R}_2)}\frac{f(\varphi(P))}{\vert
P\vert}=-\log(\sum_{T\in
\CMcal{T}(\CMcal{B}_2,\CMcal{R}_2)}(-1)^{\vert
T\vert}f(\varphi(T))).
\end{equation}
Let $\CMcal{T}_n(\CMcal{B}_2,\CMcal{R}_2)$ be the set of trivial heaps of size $n$ contained in the set $\CMcal{T}(\CMcal{B}_2,\CMcal{R}_2)$ and suppose that $T_n\in\CMcal{T}_n(\CMcal{B}_2,\CMcal{R}_2)$ is a trivial heap
having $n$ pieces $\{i_1,\ldots,i_r\}$, $\{i_{r+1},\ldots,i_{2r}\}$,
$\ldots,\{i_{rn-r+1},\ldots,i_{rn}\}$ such that $i_1<i_{r+1}<\cdots<i_{rn-r+1}$. Then $i_1,i_2,\ldots,i_{rn}$ must be a strictly increasing sequence. We use $I_{T_n}=i_1,i_2,\ldots,i_{rn}$ to denote this strictly increasing
sequence. Thus, the map $T_n\mapsto I_{T_n}$ is a bijection between $\CMcal{T}_n(\CMcal{B}_2,\CMcal{R}_2)$ and the set of strictly
increasing sequences of length $rn$. This yields
\begin{align*}
\sum_{T_n\in\CMcal{T}_n(\CMcal{B}_2,\CMcal{R}_2)}
(-1)^nf(\varphi(T_n))=\sum_{I_{T_n}}(-1)^n\prod_{j=1}^{rn}a_{i_j}
=(-1)^n\frac{(a_1+a_2+\cdots)^{rn}}{(rn)!}
\end{align*}
where the second summation runs over all the strictly increasing
sequences $I_{T_n}$ of length $rn$ and the last equation holds
because $a_i^2=0$ for every $i$. It follows immediately that
the right hand side of (\ref{E:logpy1}) is
\begin{align*}
-\log(\sum_{T\in \CMcal{T}(\CMcal{B}_2,\CMcal{R}_2)}(-1)^{\vert
T\vert}f(\varphi(T)))&=-\log(\sum_{n\ge
0}\sum_{T_n\in \CMcal{T}_n(\CMcal{B}_2,\CMcal{R}_2)}
(-1)^nf(\varphi(T_n)))\\
&=-\log(\sum_{n\ge
0}(-1)^n\frac{(a_1+a_2+\cdots)^{rn}}{(rn)!}).
\end{align*}
On the other hand, let $\CMcal{P}_n(\CMcal{B}_2,\CMcal{R}_2)$ be the set
of pyramids of size $n$ contained in the set
$\CMcal{P}(\CMcal{B}_2,\CMcal{R}_2)$. Suppose that $P_n\in
\CMcal{P}_n(\CMcal{B}_2,\CMcal{R}_2)$ is a pyramid of size $n$ such that
$f(\varphi(P_{n}))\ne 0$, then the elements of the poset
$P_n$ are exactly the blocks of a set partition of $rn$ integers where every block has size $r$, and every block is crossing with at least one other block of this set partition. We continue using $I_{T_n}$ to
represent any strictly increasing sequence $i_1,i_2,\ldots,i_{rn}$.
A strictly increasing sequence $I_{T_n}$ uniquely corresponds with a set $\{i_1,i_2,\ldots,i_{rn}\}$. For a given set
$\{i_1,i_2,\ldots,i_{rn}\}$, let $p_{n,r}$ count the number of
pyramids $P_n$ such that the elements of $P_n$ are exactly the blocks of a set partition of
$\{i_1,i_2,\ldots,i_{rn}\}$ where each block has size $r$.
Then we have
\begin{align*}
\sum_{P_n\in \CMcal{P}_n
(\CMcal{B}_2,\CMcal{R}_2)}f(\varphi(P_n))
&=\sum_{I_{T_n}}p_{n,r}\prod_{j=1}^{rn}a_{i_j}
=\sum_{I_{T_n}}p_{n,r}\frac{(a_1+a_2+\cdots)^{rn}}{(rn)!}.
\end{align*}
Note that the number $p_{n,r}$ is independent of the choice of the set $\{i_1,i_2,\ldots,i_{rn}\}$. Without loss of generality, we choose the set
$\{1,2,\ldots,rn\}$. It turns out that the left hand
side of (\ref{E:logpy1}) is
\begin{align*}
\sum_{P\in
\CMcal{P}(\CMcal{B}_2,\CMcal{R}_2)}\frac{f(\varphi(P))}{\vert P\vert}&=\sum_{n\ge 0}\sum_{P_n\in \CMcal{P}_n
(\CMcal{B}_2,\CMcal{R}_2)}
\frac{f(\varphi(P_n))}{n}=\sum_{n\ge
0}\frac{p_{n,r}}{n}\frac{(a_1+a_2+\cdots)^{rn}}{(rn)!}
\end{align*}
In view of (\ref{E:rQ2}), we get $r_n(\Pi_n^{(r)})=n^{-1}p_{n,r}$.
Since every pyramid in $\CMcal{P}(\CMcal{B}_2,\CMcal{R}_2)$ has $n$ elements, it follows that $r_n(\Pi_n^{(r)})$ counts the number of pyramids $P^*_n$ in $\CMcal{P}(\CMcal{B}_2,\CMcal{R}_2)$ such that the unique maximal element of $P_n^*$ is the block containing $rn$ and $P_n^*$ has $n$ elements which are exactly the blocks of a partition in $\Pi_{n,r}$. The proof is complete.
\end{proof}
For example, $r_2(\Pi_2^{(2)})=2$ counts the pyramid $\{\{1,3\},\{2,4\}\}$ such that $\{1,3\}\le \{2,4\}$ and the pyramid $\{\{1,4\},\{2,3\}\}$ such that $\{2,3\}\le \{1,4\}$.
\begin{lemma}\label{L:bijrn}
There is a bijection between the set of permutations of $[rn-1]$ with descent set
$\{r,2r,\ldots,rn-r\}$ and the set of pyramids $P_n^*$ in $\CMcal{P}(\CMcal{B}_2,\CMcal{R}_2)$ such that
$P_n^*$ has $n$ elements which are exactly the blocks of a partition in $\Pi_{n,r}$ and
the unique maximal element of $P_n^*$ is the block containing $rn$.
\end{lemma}
\begin{proof}
Let $\pi=a_1a_2\cdots a_{rn-1}$ be a permutation of $[rn-1]$ such that
$\mbox{Des}\,\pi=\{r,2r,\ldots,rn-r\}$. We consider a partition $\sigma\in \Pi_{n,r}$ whose blocks are exactly $$\{a_1,\ldots,a_r\},\{a_{r+1},\ldots,a_{2r}\},\cdots,\{a_{rn-r+1},\ldots,a_{rn-1},rn\},$$
and we now prove that we can choose them in such a way that they are the elements of a pyramid $P_n^*$. It remains to define an order $\le$ of the elements in the pyramid $P_n^*$. The correspondence ${\sf g}(\sigma)$ is defined inductively. Let $m$ be the minimal integer such that $m+1,m+2,\ldots,rn$ are in the same block. We notice that $m$ and $rn$ are in two crossing blocks of $\sigma$ and we consider the block that contains $m$. Suppose that $\{a_{i_1},\ldots,a_{i_1+r-1}\}$ is the block of $\sigma$ such that $a_{i_1+r-1}=m$ and let $\sigma_1=\{\{a_{k},\ldots,a_{r+k-1}\}:k\le i_1\}$ be a subset of the blocks of $\sigma$. Let $\sigma_2=\{\{a_{k},\ldots,a_{r+k-1}\}:k>i_1\}$ be the set of remaining blocks. Then we can write $\sigma=\sigma_1\sigma_2$ and we use $\vert\sigma_i\vert$ to denote the number of blocks in the partition $\sigma_i$, for $i=1,2$. By induction, ${\sf g}(\sigma_i)$ is a pyramid such that the unique maximal element is the block containing $m$ if $i=1$, or the block containing $rn$ if $i=2$, and ${\sf g}(\sigma_i)$ has $\vert\sigma_i\vert$ elements which are exactly the blocks of $\sigma_i$, for $i=1,2$. For every element $x$ in the pyramid ${\sf g}(\sigma_1)$, suppose that $y_x$ is one of the minimal elements in the pyramid ${\sf g}(\sigma_2)$ that is crossing with $x$ but not crossing with any $z$ such that $x\le z$ in the pyramid ${\sf g}(\sigma_1)$, then the pyramid ${\sf g}(\sigma)$ is obtained by letting $y_x$ cover $x$ for every $x$ and $y_x$. In particular, $\{a_{i_1},\ldots,a_{i_1+r-1}\}\le \{a_{rn-r+1},\ldots,a_{rn-1},rn\}$ in the pyramid ${\sf g}(\sigma)$, which implies that the unique maximal element of ${\sf g}(\sigma)$ is the block containing $rn$. In fact, this provides an inductive process to successively construct the pyramid ${\sf g}(\sigma)$ whose unique maximal element is the block of $\sigma$ containing $rn$ and ${\sf g}(\sigma)$ has $n$ elements that are exactly the blocks of $\sigma$.

Conversely, consider a pyramid $P_n^*$ in $\CMcal{P}(\CMcal{B}_2,\CMcal{R}_2)$ such that $P_n^*$ has $n$ elements which are exactly the blocks of a partition $\sigma\in\Pi_{n,r}$ and the unique maximal element is $\{a_{rn-r+1},\ldots,a_{rn-1},rn\}$. We consider the pyramid $P_1$ that is induced by the block that contains $m$, i.e., $P_1$ contains all the elements $x$ such that $x\le \{a_{i_1},\ldots,a_{i_1+r-1}\}$ in the pyramid $P_n^*$ where $a_{i_1+r-1}=m$. Let $P_2$ be the pyramid of remaining elements from $P_n^*$. Then $P_i={\sf g}(\sigma_i)$ for $i=1,2$ which implies the above inductive process is bijective and therefore the proof is complete.
\end{proof}
See Figure~\ref{F:4} for an example of this bijection ${\sf g}$.
\begin{figure}[htbp]
\begin{center}
\includegraphics[scale=0.7]{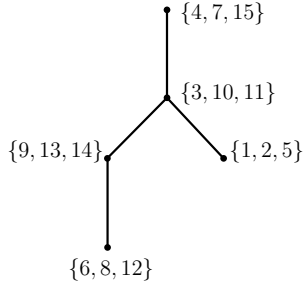}
\caption{The pyramid ${\sf g}(\sigma)$ where $\pi=6\,8\,12\,9\,13\,14\,1\,2\,5\,3\,10\,11\,4\,7\in\mathfrak{S}_{14}$ and $\sigma$ is a partition with blocks $\{6,8,12\}$, $\{9,13,14\}$, $\{1,2,5\}$, $\{3,10,11\}$, $\{4,7,15\}$.
\label{F:4}}
\end{center}
\end{figure}
\section{On the poset ${\sf Q}_n^{(r)}$}\label{S:rpar}
An $r$-partition of $[n]$ is a set
\begin{align*}
\pi=\{(B_{11},\cdots,B_{1r}),(B_{21},\cdots,B_{2r}),\cdots,
(B_{k1},\cdots, B_{kr})\}
\end{align*}
satisfying the following two conditions:
\begin{enumerate}
\item for each $j\in [r]$, the set
$\pi_j=\{B_{1j},B_{2j},\cdots,B_{kj}\}$ forms a partition of $S$
(into $k$ blocks),
\item for fixed $i$, $\vert B_{i1}\vert=\vert B_{i2}\vert
=\cdots=\vert B_{ir}\vert$.
\end{enumerate}
The set of all the $r$-partitions of set $[n]$, denoted by
${\sf Q}_n^{(r)}$, has a partial ordering by refinement. Namely, let
\begin{align*}
\pi&=\{(B_{11},\cdots,B_{1r}),(B_{21},\cdots,B_{2r}),\cdots,
(B_{k1},\cdots, B_{kr})\}\\
\sigma&=\{(A_{11},\cdots,A_{1r}),(A_{21},\cdots,A_{2r}),\cdots,
(A_{\ell 1},\cdots, A_{\ell r})\}
\end{align*}
be two $r$-partitions of $[n]$ where we set for $1\le j\le r$,
\begin{align*}
\pi_j=\{B_{1j},B_{2j},\cdots,B_{kj}\}\quad \mbox{ and }\quad
\sigma_j=\{A_{1j},A_{2j},\cdots,A_{\ell j}\}.
\end{align*}
Then $\pi\le\sigma$ if for every $1\le j\le r$, we have $\pi_j\le \sigma_j$,
i.e., every block of partition $\pi_j$ is contained in a block of
partition $\sigma_j$. For instance,
$$\pi=\{(\{1\},\{2\}),(\{2\},\{3\}),(\{3\},\{1\})\}\quad \mbox{ and }\quad\sigma=\{(\{1,2\},\{2,3\}),(\{3\},\{1\})\}$$
are two $2$-partitions
of $[3]$. Then we have $\pi_1=\pi_2=\{\{1\},\{2\},\{3\}\}$,
$\sigma_1=\{\{1,2\},\{3\}\}$, $\sigma_2=\{\{2,3\},\{1\}\}$ and
$\pi_1\le \sigma_1$, $\pi_2\le \sigma_2$ by refinement. That implies that
$\pi\le\sigma$ in the poset ${\sf Q}_3^{(2)}$.
By definition, every minimal element of ${\sf Q}_n^{(r)}$ can be identified as an $(r-1)$-tuple
$(\sigma_1,\sigma_2,\ldots,\sigma_{r-1})$ of permutations
$\sigma_i\in\mathfrak{S}_n$ by
\begin{align}\label{E:min}
\rho=\{(\{1\},\{\sigma_1(1)\},\ldots,\{\sigma_{r-1}(1)\}),\ldots,
(\{n\},\{\sigma_1(n)\},\ldots,\{\sigma_{r-1}(n)\})\}.
\end{align}
As an immediate consequence, the number of minimal elements of ${\sf Q}_n^{(r)}$ is equal to the number of ordered permutations of $[n]$, i.e.,
$(\sigma_1,\sigma_2,\cdots,\sigma_{r-1})$ where
$\sigma_i\in\mathfrak{S}_n$. Let ${\sf Q}_{n,r}$ be the set of $(r-1)$-tuple $(\sigma_1,\sigma_2,\cdots,\sigma_{r-1})$ of permutations $\sigma_i\in\mathfrak{S}_n$,
it follows that $M(n)=\vert {\sf Q}_{n,r}\vert=n!^{r-1}$ and that the sequence $\{r_n({\sf Q}_n^{(r)})\}_{n\ge 1}$ where $r\ge 2$ defined by
(\ref{E:defr}) satisfies
\begin{eqnarray}\label{E:rQ3}
\sum_{n\ge 1}\frac{r_n({\sf Q}_n^{(r)})z^n}{n!^{r}}=-\log(\sum_{n\ge
0}(-1)^n\frac{z^n}{n!^{r}}).
\end{eqnarray}
For any $\sigma=(\sigma_1,\sigma_2,\ldots,\sigma_{r-1})\in {\sf Q}_{n,r}$, the diagram representation of $\sigma$ is described as follows: We draw $rn$ dots in $r$ rows with each row having $n$ dots labeled by $1,2,\ldots,n$, and, for every $i,m$, we connect $m$ from the $i$-th row with $\sigma_i(m)$ from the $(i+1)$-th row. For every $m\in [n]$, we call the sequence $$(m,\sigma_1(m),\sigma_2\sigma_1(m),\ldots,\sigma_{r-1}\cdots\sigma_2\sigma_1(m))$$
a {\em path} of $\sigma$ that starts from $m$. See Figure~\ref{F:5}. Clearly every $\sigma\in {\sf Q}_{n,r}$ has $n$ paths. Two paths $(a_1,a_2,\ldots,a_r)$ and $(b_1,b_2,\ldots,b_r)$ are not crossing if and only if $a_k<b_k$ for all $k$. Otherwise two paths $(a_1,a_2,\ldots,a_r)$ and $(b_1,b_2,\ldots,b_r)$ are crossing.
\begin{figure}[htbp]
\begin{center}
\includegraphics[scale=0.7]{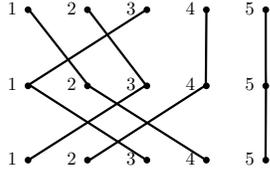}
\caption{The diagram representation of $\sigma=(23145,34125)$ and $\sigma$ has paths
$(1,2,4)$, $(2,3,1)$, $(3,1,3)$, $(4,4,2)$ and $(5,5,5)$.
\label{F:5}}
\end{center}
\end{figure}
Let $\CMcal{B}_3=\{(i_1,i_2,\ldots,i_r):i_j\ge 1, 1\le j\le r\}$ be the set of paths of length $r-1$ with a symmetric and reflexive binary relation $\CMcal{R}_3$ defined by: $(i_1,\ldots,i_r)\CMcal{R}_3(j_1,\ldots,j_r)$ if and only if two paths
$(i_1,\ldots,i_r)$ and $(j_1,\ldots,j_r)$ are crossing. Let
$(R,\le)$ be a poset where each element is labeled by an element
$(i_1,\ldots,i_r)\in\CMcal{B}_3$ such that:
\begin{enumerate}
\item $(i_1,\ldots,i_r)$ and
$(j_1,\ldots,j_r)$ are comparable if $(i_1,\ldots,i_r)$
and $(j_1,\ldots,j_r)$ are crossing.
\item if $(j_1,\ldots,j_r)$ covers $(i_1,\ldots,i_r)$ in the poset $(R,\le)$, then
they are crossing.
\end{enumerate}
By definition~\ref{D:heap}, the poset $(R,\le)$ is a heap $H=(R,\le)\in\CMcal{H}(\CMcal{B}_3,\CMcal{R}_3)$. In a geometric way, we
put the path $(j_1,\ldots,j_r)$ on top of the path $(i_1,\ldots,i_r)$ if
$(i_1,\ldots,i_r)\le (j_1,\ldots,j_r)$ in the heap $H=(R,\le)$. See Figure~\ref{F:6}.
\begin{figure}[htbp]
\begin{center}
\includegraphics[scale=0.7]{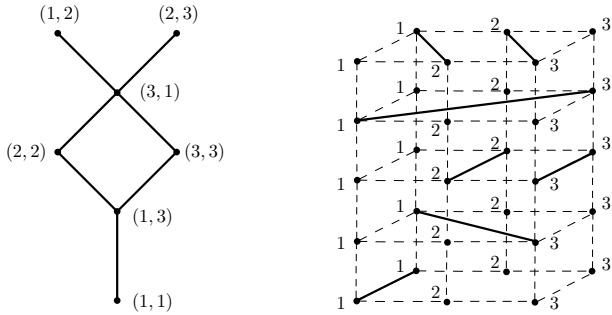}
\caption{A heap $H$ (left) in the monoid $\CMcal{H}(\CMcal{B}_3,\CMcal{R}_3)$ and its geometric representation (right). The corresponding word in the monoid $\CMcal{M}(\CMcal{B}_3,\CMcal{R}_3)$ is $\varphi(H)=x_{(1,1)}x_{(1,3)}x_{(2,2)}x_{(3,3)}x_{(3,1)}x_{(1,2)}x_{(2,3)}$.
\label{F:6}}
\end{center}
\end{figure}

Furthermore, {\em non-ambiguous trees} were introduced by Aval {\it et
al.} \cite{A:13} because of their connection to the tree-like tableaux. The non-ambiguous trees are embedded in a $2$-dimensional $\mathbb{N}\times\mathbb{N}$ grid. Let each vertex $v$ have coordinates
$x(v)=(x_1(v),x_2(v))$. Then a $2$-dimensional non-ambiguous tree of
size $n$ is a set $A$ of $n$ points
$(x_1(v),x_2(v))\in\mathbb{N}\times\mathbb{N}$ such that:
\begin{enumerate}
\item $(0,0)\in A$; we call this point the root of $A$.
\item for a given non-root point $p\in A$, there exists one point $q\in
A$ such that $x_2(q)<x_2(p)$ and $x_1(q)=x_1(p)$ or one point $s\in
A$ such that $x_1(s)<x_1(p)$ and $x_2(s)=x_2(p)$ (but not both).
\item there is no empty line between two given points: if there
exists a point $p\in A$ such that $x_1(p)=x$ (resp. $x_2(p)=y$),
then for every $x'<x$ (resp. $y'<y$) there exists $q\in A$ such that
$x_1(q)=x'$ (resp. $x_2(q)=y'$).
\end{enumerate}
A {\em complete non-ambiguous tree} is a non-ambiguous tree whose vertices
have either $0$ or $2$ children. The non-ambiguous tree $A$ has a
unique tree structure since, except for the root, every point $p\in
A$ has a unique parent, which is the nearest point $q$ or $s$ in
condition $(2)$. In other words, the set $A$ of points determines
the tree structure. Let $T_A$ be the unique underlying tree
associated to the vertex set $A$. For instance, there are four
complete non-ambiguous trees of size $5$ whose underlying trees are given here.
\begin{center}
\setlength{\unitlength}{3.5pt}
\begin{picture}(30,18)(32,-14)
\put(10,0){\circle*{0.8}}\put(10,0){\line(-1,-1){8}}
\put(10,0){\line(1,-1){8}}\put(2,-8){\circle*{0.8}}
\put(14,-4){\circle*{0.8}}\put(18,-8){\circle*{0.8}}
\put(14,-4){\line(-1,-1){4}}\put(10,-8){\circle*{0.8}}
\put(7,1){\scriptsize{$(0,0)$}}\put(-2,-5){\scriptsize{$(0,2)$}}
\put(14,-3){\scriptsize{$(1,0)$}}\put(6,-11){\scriptsize{$(1,1)$}}
\put(18,-7){\scriptsize{$(2,0)$}}
\put(35,0){\circle*{0.8}}\put(35,0){\line(-1,-1){8}}
\put(35,0){\line(1,-1){8}}\put(27,-8){\circle*{0.8}}
\put(31,-4){\circle*{0.8}}\put(43,-8){\circle*{0.8}}
\put(31,-4){\line(1,-1){4}}\put(35,-8){\circle*{0.8}}
\put(32,1){\scriptsize{$(0,0)$}}\put(26,-11){\scriptsize{$(0,2)$}}
\put(25,-3){\scriptsize{$(0,1)$}}\put(33,-11){\scriptsize{$(1,1)$}}
\put(43,-7){\scriptsize{$(2,0)$}}
\put(60,0){\circle*{0.8}}\put(60,0){\line(-1,-1){8}}
\put(60,0){\line(1,-1){4}}\put(64,-4){\circle*{0.8}}
\put(56,-4){\circle*{0.8}}\put(52,-8){\circle*{0.8}}
\put(56,-4){\line(1,-1){8}}\put(64,-12){\circle*{0.8}}
\put(57,1){\scriptsize{$(0,0)$}}\put(50,-3){\scriptsize{$(0,1)$}}
\put(64,-3){\scriptsize{$(1,0)$}}\put(50,-12){\scriptsize{$(0,2)$}}
\put(57,-12){\scriptsize{$(2,1)$}}
\put(85,0){\circle*{0.8}}\put(85,0){\line(-1,-1){4}}
\put(85,0){\line(1,-1){8}}\put(89,-4){\circle*{0.8}}
\put(81,-4){\circle*{0.8}}\put(93,-8){\circle*{0.8}}
\put(89,-4){\line(-1,-1){8}}\put(81,-12){\circle*{0.8}}
\put(82,1){\scriptsize{$(0,0)$}}\put(75,-3){\scriptsize{$(0,1)$}}
\put(89,-3){\scriptsize{$(1,0)$}}\put(93,-7){\scriptsize{$(2,0)$}}
\put(83,-12){\scriptsize{$(1,2)$}}
\end{picture}
\end{center}
One of our main results is stated
as follows:
\begin{theorem}\label{T:rnQn}
$r_n({\sf Q}_n^{(r)})$ counts the number of pyramids $Q_n^*$ in $\CMcal{P}(\CMcal{B}_3,\CMcal{R}_3)$ such that
$Q_n^*$ has $n$ elements which are exactly the paths of an $(r-1)$-tuple of permutations in $\mathfrak{S}_{n}$ and
the unique maximal element of $Q_n^*$ is the path starting from $1$. In particular,
$r_n({\sf Q}_n^{(2)})$ counts the number of complete non-ambiguous trees of size $2n-1$.
\end{theorem}
We divide Theorem~\ref{T:rnQn} into Lemma~\ref{L:bcm} and Lemma~\ref{L:bicm}. We shall prove Lemma~\ref{L:bcm} again by using the connection between $r_n({\sf Q}_n^{(r)})$ and heaps and prove Lemma~\ref{L:bicm} by a bijection.
\subsection{Connection between $r_n({\sf Q}_n^{(r)})$ and heaps}\label{s:type2notation}

Now we first establish a connection between $r_n({\sf Q}_n^{(r)})$ and the set of pyramids $\CMcal{P}(\CMcal{B}_3,\CMcal{R}_3)$ via (\ref{E:logpy}).
\begin{lemma}\label{L:bcm}
$r_n({\sf Q}_n^{(r)})$ counts the number of pyramids $Q_n^*$ in $\CMcal{P}(\CMcal{B}_3,\CMcal{R}_3)$ such that $Q_n^*$ has $n$ elements which are exactly the paths of an $(r-1)$-tuple of permutations in $\mathfrak{S}_{n}$ and the unique maximal element of $Q_n^*$ is the path starting from $1$.
\end{lemma}
\begin{proof}
From (\ref{E:logpy}) we know
\begin{equation}\label{E:c1}
\sum_{P\in \CMcal{P}(\CMcal{B}_3,\CMcal{R}_3)}\frac{\varphi(P)}{\vert P\vert} =_{\mbox{\scriptsize{comm}}} -\log(\sum_{T\in
\CMcal{T}(\CMcal{B}_3,\CMcal{R}_3)}(-1)^{\vert T\vert}\varphi(T)).
\end{equation}
We now consider the Cartier-Foata monoid $\CMcal{M}(\CMcal{B}_3,\CMcal{R}_3)$.
Let $\{a_{i,j}\}_{i\ge 1,j\ge 1}$ be a sequence of variables such that
$a_{i,j}^2=0$ for every $i,j$ and $a_{i,j}a_{k,\ell}=a_{k,\ell}a_{i,j}$ for every $i,j,k,\ell$. Furthermore let $h:\mathbb{Z}[[\CMcal{M}(\CMcal{B}_3,\CMcal{R}_3)]]\rightarrow
\mathbb{Z}[[a_{1,1},a_{1,2},\ldots]]$ be a ring homomorphism such that
$h(x_{(i_1,i_2,\ldots,i_r)})=\prod_{j=1}^r a_{i_j,j}$ and $h(1)=1$. We apply $h$ on
both sides of (\ref{E:c1}) and obtain
\begin{equation}\label{E:c2}
\sum_{P\in \CMcal{P}(\CMcal{B}_3,\CMcal{R}_3)}\frac{h(\varphi(P))}{\vert P\vert}=-\log(\sum_{T\in
\CMcal{T}(\CMcal{B}_3,\CMcal{R}_3)}(-1)^{\vert T \vert}h(\varphi(T))).
\end{equation}
Let $\CMcal{T}_n(\CMcal{B}_3,\CMcal{R}_3)$ be the set of trivial heaps of
size $n$ contained in the set $\CMcal{T}(\CMcal{B}_3,\CMcal{R}_3)$ and suppose that $T_n\in\CMcal{T}_n(\CMcal{B}_3,\CMcal{R}_3)$ is a trivial heap
having $n$ paths $(i_1,\ldots,i_r),(i_{r+1},\ldots,i_{2r})$,
$\ldots$, $(i_{rn-r+1},\ldots,i_{rn})$ such that $i_1<i_{r+1}<\cdots<i_{rn-r+1}$. Then
the sequence $I_{T_n,j}=i_j,i_{r+j},\ldots,i_{nr-r+j}$ is a strictly increasing sequence for every $j$, $1\le j\le r$. It follows that $T_n\mapsto I_{T_n}^*$ where
$I_{T_n}^*=(I_{T_n,1},\ldots,I_{T_n,r})$ and
$I_{T_n,j}=i_j,i_{r+j},\ldots,i_{nr-r+j}$ is a bijection between the set
of trivial heaps of size $n$ and the set of $r$-tuples of strictly
increasing sequences of length $n$. This gives
\begin{align*}
\sum_{T_n\in\CMcal{T}_n(\CMcal{B}_3,\CMcal{R}_3)}
(-1)^nh(\varphi(T_n))&=\sum_{I_{T_n}^*}(-1)^n\prod_{m=1}^{n}
\prod_{j=1}^{r}a_{i_{rm-r+j},j}\\
&=(-1)^n\prod_{j=1}^{r}\frac{(a_{1,j}+a_{2,j}+\cdots)^{n}}{n!}
=\frac{(-1)^n}{n!^r}(\prod_{j=1}^{r}(\sum_{i=1}^{\infty}
a_{i,j}))^{n}
\end{align*}
where the second summation runs over all the $r$-tuples $I_{T_n}^*$
of strictly increasing sequences of length $n$ and the last two
equations hold because $a_{i,j}^2=0$ for every $i,j$. It follows immediately
that the right hand side of (\ref{E:c2}) is
\begin{align*}
-\log(\sum_{T\in \CMcal{T}(\CMcal{B}_3,\CMcal{R}_3)}(-1)^{\vert T\vert}h(\varphi(T)))&=-\log(\sum_{n\ge
0}\sum_{T_n\in \CMcal{T}_n(\CMcal{B}_3,\CMcal{R}_3)}
(-1)^nh(\varphi(T_n)))\\
&=-\log(\sum_{n\ge
0}\frac{(-1)^n}{n!^r}(\prod_{j=1}^{r}(\sum_{i=1}^{\infty}
a_{i,j}))^{n}).
\end{align*}
On the other hand, let $\CMcal{P}_n(\CMcal{B}_3,\CMcal{R}_3)$ be the set
of pyramids of size $n$ contained in the set
$\CMcal{P}(\CMcal{B}_3,\CMcal{R}_3)$. Suppose that $P_n\in
\CMcal{P}_n(\CMcal{B}_3,\CMcal{R}_3)$ is a pyramid of size $n$ such that
$h(\varphi(P_{n}))\ne 0$, then the elements of the poset $P_n$ are exactly the paths of an $(r-1)$-tuple $\sigma=(\sigma_1,\ldots,\sigma_{r-1})$ of bijections such that $\sigma_j$ is a bijection between two
sets of $n$ integers for every $j$, and every path of $\sigma$ is crossing with at least one other path of $\sigma$. We continue
using $I_{T_n}^*=(I_{T_n,1},\ldots,I_{T_n,r})$ to represent any
$r$-tuple of strictly increasing sequences of length $n$. A strictly
increasing sequence $I_{T_n,j}=i_{j},i_{j+r},\ldots,i_{nr-r+j}$ uniquely corresponds with a set $S_{T_n,j}=\{i_{j},i_{j+r},\ldots,i_{nr-r+j}\}$. For a given $r$-tuple $I_{T_n}^*$ of strictly increasing sequences, let
$\bar{p}_{n,r}$ count the number of pyramids $P_n$ such that the elements of $P_n$ are exactly the paths of an $(r-1)$-tuple $\sigma=(\sigma_1,\ldots,\sigma_{r-1})$ of bijections such that $\sigma_j$ is a bijection between the set $S_{T_n,j}$ and $S_{T_n,j+1}$ for every $j$. Then we have
\begin{align*}
\sum_{P_n\in
\CMcal{P}_n(\CMcal{B}_3,\CMcal{R}_3)}h(\varphi(P_n))=\sum_{I_{T_n}^*}\bar{p}_{n,r}
\prod_{m=1}^{n}\prod_{j=1}^{r}a_{i_{rm-r+j},j}
\end{align*}
Note that the number $\bar{p}_{n,r}$ is independent of the choice
of the $r$-tuple $I_{T_n}^*$ of strictly increasing sequences and we choose the set $S_{T_n,j}=\{1,2,\ldots,n\}$ for every $j$. It turns out the left hand side of (\ref{E:c2}) is
\begin{align*}
\sum_{P\in \CMcal{P}(\CMcal{B}_3,\CMcal{R}_3)}\frac{h(\varphi(P))}{\vert P\vert}&=\sum_{n\ge 0}\sum_{P_n\in
\CMcal{P}_n(\CMcal{B}_3,\CMcal{R}_3)}
\frac{h(\varphi(P_n))}{n}=\sum_{n\ge 0}\sum_{I_{T_n}^*}\frac{\bar{p}_{n,r}}{n}
\prod_{m=1}^{n}\prod_{j=1}^{r}a_{i_{rm-r+j},j}\\
&=\sum_{n\ge 0}\frac{\bar{p}_{n,r}}{n}\frac{1}{n!^r}
(\prod_{j=1}^{r}(\sum_{i=1}^{\infty} a_{i,j}))^{n}.
\end{align*}
In view of (\ref{E:rQ3}), we conclude that
$r_n({\sf Q}_n^{(r)})=n^{-1}\bar{p}_{n,r}$.
Since every pyramid in $\CMcal{P}(\CMcal{B}_3,\CMcal{R}_3)$ has $n$ elements, it follows that $r_n({\sf Q}_n^{(r)})$ counts the number of pyramids $Q_n^*$ in $\CMcal{P}(\CMcal{B}_3,\CMcal{R}_3)$ such that $Q_n^*$ has $n$ elements which are exactly the paths of an $(r-1)$-tuple of permutations in $\mathfrak{S}_{n}$ and the maximal element of $Q_n^*$ is the path starting at $1$.
The proof is complete.
\end{proof}
For example, $r_2({\sf Q}_2^{(3)})=3$ counts the pyramid $\{(1,2,1),(2,1,2)\}$ such that $(2,1,2)\le(1,2,1)$, the pyramid $\{(1,2,2),(2,1,1)\}$ such that $(2,1,1)\le (1,2,2)$, and the pyramid $\{(1,1,2),(2,2,1)\}$ such that $(2,2,1)\le (1,1,2)$.

Let $b_{n}$ be the number of complete non-ambiguous trees of size
$2n-1$, Aval {\it et al.} proved the integers $b_{n}$ satisfy
(\ref{E:rQ3}) when $r=2$. This implies that $r_n({\sf Q}_n^{(2)})=b_{n}$.
We will prove that $r_n({\sf Q}_n^{(2)})=b_{n}$ by a bijection in Lemma~\ref{L:bicm}.
\begin{lemma}\label{L:bicm}
There is a bijection between the set of complete non-ambiguous
trees of size $2n-1$ and the set of pyramids $Q_n^*$ in $\CMcal{P}(\CMcal{B}_3,\CMcal{R}_3)$ such that
$Q_n^*$ has $n$ elements which are exactly the paths of a permutation $\sigma\in \mathfrak{S}_n$ and the unique
maximal element of $Q_n^*$ is the path $(1,\sigma(1))$.
\end{lemma}
\begin{proof}
Given a complete non-ambiguous tree $A$ of size $2n-1$, we first do
a coordinate translation. Let $A\mapsto \mu(A)$ be a bijection such that
for any $v\in A$ having coordinate $x(v)=(x_1(v),x_2(v))$, we set
$\mu(A)=\{(x_1(v)+1,x_2(v)+1):v\in A\}$. Recall that $T_A$ is the
unique underlying tree associated to $A$. Thus via bijection $\mu$, we shift the tree $T_A$ rooted at $(0,0)$ to the tree
$T_{\mu(A)}$ rooted at $(1,1)$. Next we consider the coordinates
of the leaves of $T_{\mu(A)}$. From the definition of complete
non-ambiguous tree, we know the coordinates of the leaves of
$T_{\mu(A)}$ must be
$$(1,\sigma(1)),(2,\sigma(2))\cdots,(n,\sigma(n)),$$
where $\sigma\in\mathfrak{S}_n$ and we now prove that we can choose them in such a way that they are the elements of a pyramid
$Q_n^*$. It remains to define an order $\le$ of the elements in the pyramid $Q_n^*$.

Let $B$ be any subset of $A$ such that the underlying tree $T_{\mu(B)}$ is a subtree of $T_{\mu(A)}$. The correspondence ${\sf q}(B)$ is defined inductively. Suppose $B$ is rooted at $(a,b)$ and let $B_1,B_2$ be two subsets of $B$ such that $T_{\mu(B_1)}$ is the left subtree and $T_{\mu(B_2)}$ is the right subtree of $T_{\mu(B)}$. Then for $i=1,2$, the leaves of $T_{\mu(B_i)}$ can be identified as a permutation $\bar{\sigma}_i$ which is $\sigma$ restricted to a subset of $[n]$. Clearly $(a,\sigma(a))\in B_1$, $(\sigma^{-1}(b),b)\in B_2$ and the two paths $(a,\sigma(a)),(\sigma^{-1}(b),b)$ are crossing. By induction, ${\sf q}(B_1)$ is a pyramid such that the unique maximal element is $(a,\sigma(a))$, ${\sf q}(B_2)$ is a pyramid such that the unique maximal element is $(\sigma^{-1}(b),b)$ and all the elements of ${\sf q}(B_i)$ are exactly the paths of $\bar{\sigma}_i$ for $i=1,2$. We use $\vert \bar{\sigma}_i\vert$ to denote the number of paths in $\bar{\sigma}_i$ for $i=1,2$.

If $T_{\mu(B)}$ is the left subtree of its parent in $T_{\mu(A)}$ or $B=A$, then, for every element $x$ in the pyramid ${\sf q}(B_2)$, suppose that $y_x$ is one of the minimal elements in the pyramid ${\sf q}(B_1)$ that is crossing with $x$ but not crossing with any $z$ such that $x\le z$ in the pyramid ${\sf q}(B_2)$. Thus the pyramid ${\sf q}(B)$ is obtained by letting $y_x$ cover $x$ for every $x$ and $y_x$. In particular, $(\sigma^{-1}(b),b)\le (a,\sigma(a))$ holds in the pyramid ${\sf q}(B)$ which implies that the unique maximal element of ${\sf q}(B)$ is $(a,\sigma(a))$.

If $T_{\mu(B)}$ is the right subtree of its parent in $T_{\mu(A)}$, then, for every path $y$ in the pyramid ${\sf q}(B_1)$, suppose $x_y$ is one of the minimal elements in the pyramid ${\sf q}(B_2)$ that is crossing with $y$ but not crossing with any $z$ such that $y\le z$ in the pyramid ${\sf q}(B_1)$. Thus, the pyramid ${\sf q}(B)$ is obtained by letting $x_y$ cover $y$ for every $y$ and $x_y$. In particular, $(a,\sigma(a))\le(\sigma^{-1}(b),b)$ holds in the pyramid ${\sf q}(B)$ which implies that the unique maximal element of ${\sf q}(B)$ is $(\sigma^{-1}(b),b)$.

In fact, this gives us an inductive process to successively construct the pyramid ${\sf q}(A)$ whose unique maximal element is $(1,\sigma(1))$ and ${\sf q}(A)$ has $n$ elements which are exactly the paths of $\sigma\in \mathfrak{S}_n$. Conversely, without loss of generality, consider a pyramid $Q_n^*$ such that the elements of $Q_n^*$ are exactly the paths of the permutations $\bar{\sigma}_1,\bar{\sigma}_2$ and the unique maximal element is $(a,\sigma(a))$. Let $d$ be the minimal integer such that $d_1\ge d$ for all the elements $(c_1,d_1)\in Q_n^*$,
then $d=b$ and we consider the pyramid $Q_2$ that is induced by the element $(\sigma^{-1}(b),b)$, i.e., $Q_2$ contains all the elements $x$ such that $x\le (\sigma^{-1}(b),b)$ in the pyramid $Q_n^*$. Let $Q_1$ be the pyramid of remaining elements from $Q_n^*$. Then $Q_i={\sf q}(B_i)$, for $i=1,2$, which implies the above inductive process is bijective and therefore the proof is complete.
\end{proof}
See Figure~\ref{F:7} for an example.
\begin{figure}[htbp]
\begin{center}
\includegraphics[scale=0.7]{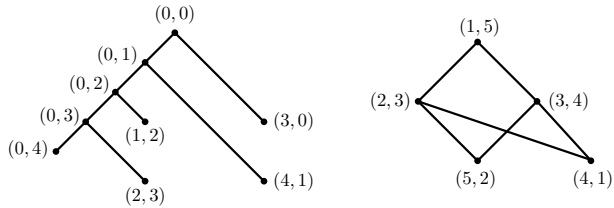}
\caption{A complete non-ambiguous tree of size $9$ (left) and the corresponding pyramid
$Q_5^*$ whose unique maximal element is the path $(1,5)$ (right). The pyramid $Q_5^*$ has $5$ elements which are exactly the paths of the permutation $\sigma=53412$.
\label{F:7}}
\end{center}
\end{figure}

\subsection{From heaps to pairs of permutations}\label{S:forAval}
A non-ambiguous forest introduced by Aval {\it et al.} \cite{A:13}
is a set of points $x(v)=(x_1(v),x_2(v))\in\mathbb{N}\times
\mathbb{N}$ satisfying the following conditions:
\begin{enumerate}
\item for a given non-root point $p\in A$, the set
$\{s\in A: x_1(s)<x_1(p),x_2(s)=x_2(p)\}\cup \{q\in A:
x_2(q)<x_2(p),x_1(q)=x_1(p)\}$ has at most $1$ element.
\item there is no empty line between two given points: if there exists
a point $p\in A$ such that $x_1(p)=x$ (resp. $x_2(p)=y$), then for
every $x'<x$ (resp. $y'<y$) there exists $q\in A$ such that
$x_1(q)=x'$ (resp. $x_2(q)=y'$).
\end{enumerate}
In the same way, a non-ambiguous forest has an underlying binary
forest structure. A complete non-ambiguous forest is a non-ambiguous
forest such that all trees in the underlying binary forest are
complete. For instance, there are $4$ complete non-ambiguous forests that
contain leaves $(0,1),(1,2),(2,0)$, which are
\begin{center}
\setlength{\unitlength}{3.5pt}
\begin{picture}(30,18)(32,-14)
\put(10,0){\circle*{0.8}}\put(10,0){\line(-1,-1){4}}
\put(10,0){\line(1,-1){8}}\put(14,-4){\circle*{0.8}}
\put(6,-4){\circle*{0.8}}\put(18,-8){\circle*{0.8}}
\put(14,-4){\line(-1,-1){8}}\put(6,-12){\circle*{0.8}}
\put(7,1){\scriptsize{$(0,0)$}}\put(0,-3){\scriptsize{$(0,1)$}}
\put(14,-3){\scriptsize{$(1,0)$}}\put(18,-7){\scriptsize{$(2,0)$}}
\put(8,-12){\scriptsize{$(1,2)$}}
\put(31,-4){\circle*{0.8}}\put(43,-8){\circle*{0.8}}
\put(31,-12){\circle*{0.8}} \put(27,-2){\scriptsize{$(0,1)$}}
\put(40,-6){\scriptsize{$(2,0)$}} \put(33,-12){\scriptsize{$(1,2)$}}
\put(64,-4){\line(1,-1){4}} \put(64,-4){\line(1,-1){4}}
\put(64,-4){\circle*{0.8}}
\put(56,-4){\circle*{0.8}}\put(68,-8){\circle*{0.8}}
\put(64,-4){\line(-1,-1){8}}\put(56,-12){\circle*{0.8}}
\put(53,-2){\scriptsize{$(0,1)$}}
\put(61,-2){\scriptsize{$(1,0)$}}\put(68,-7){\scriptsize{$(2,0)$}}
\put(58,-12){\scriptsize{$(1,2)$}}
\put(85,0){\circle*{0.8}}\put(85,0){\line(-1,-1){4}}
\put(85,0){\line(1,-1){8}}
\put(81,-4){\circle*{0.8}}\put(93,-8){\circle*{0.8}}
\put(81,-12){\circle*{0.8}}
\put(82,1){\scriptsize{$(0,0)$}}\put(75,-3){\scriptsize{$(0,1)$}}
\put(93,-7){\scriptsize{$(2,0)$}} \put(83,-12){\scriptsize{$(1,2)$}}
\end{picture}
\end{center}
We denote by $\tau(n)$ the number of complete
non-ambiguous forests with $n$ leaves. Then Aval {\it et al.}
\cite{A:13} proved that
\begin{eqnarray}\label{E:comnaf}
\sum_{n\ge 0}\tau(n)\frac{x^n}{n!^2}=(\sum_{n\ge 0}\frac{(-1)^n
x^n}{n!^2})^{-1}.
\end{eqnarray}
In addition, Carlitz {\it et al.} \cite{Carlitz} proved the same
identity by enumerating the pairs of permutations with no common
rise. Let $\omega(n)$ count the number of pairs $(\pi,\xi)$ of
permutations of $[n]$ such that any two consecutive entries cannot
be rising both in $\pi$ and $\xi$, i.e., there is no integer $i$
such that $\pi(i)<\pi(i+1)$ and $\xi(i)<\xi(i+1)$. Carlitz {\it et
al.} \cite{Carlitz} showed that
\begin{eqnarray*}
\sum_{n\ge 0}\omega(n)\frac{x^n}{n!^2}=(\sum_{n\ge
0}\frac{(-1)^n x^n}{n!^2})^{-1}.
\end{eqnarray*}
This implies that $\tau(n)=\omega(n)$ and Aval {\it et al.} asked for a bijective proof of $\tau(n)=\omega(n)$. Here we will give a bijection between the set of complete non-ambiguous forests with $n$
leaves and the set of pairs $(\pi,\xi)$ of permutations
$\pi,\xi\in\mathfrak{S}_n$ with no common rise by using heaps
as intermediate objects.
\begin{theorem}
There is a bijection between the set of complete non-ambiguous
forests with $n$ leaves and the set of pairs $(\pi,\xi)$ of
permutations of $[n]$ with no common rise.
\end{theorem}
\begin{proof}
We first define the number $h_{n,r}$ by the equation
\begin{align}\label{E:rQ4}
\sum_{n\ge 0}h_{n,r}\frac{x^n}{n!^r}=(\sum_{n\ge 0}\frac{(-1)^n
x^n}{n!^r})^{-1}.
\end{align}
Similar to the proof of Lemma~\ref{L:bcm}, via (\ref{E:CF2}) we can prove that
$h_{n,r}$ counts the number of heaps $H_n^*$ in $\CMcal{H}(\CMcal{B}_3,\CMcal{R}_3)$ such that $H_n^*$ has $n$ elements which are exactly the paths of an $(r-1)$-tuple of permutations in $\mathfrak{S}_n$. We will first show that $\tau(n)=h_{n,2}$ by a bijection.

Given a complete non-ambiguous forest $A$ with $n$ leaves, let $v\in A$ have coordinates $(x_1(v),x_2(v))$ and we again obtain the set $\mu(A)=\{(x_1(v)+1,x_2(v)+1):v\in A\}$ via
the bijection $\mu$ given in Lemma~\ref{L:bicm}.
Recall that $T_{\mu(A)}$ is the underlying forest associated with the
vertex set $\mu(A)$. From the definition we know the coordinates of the leaves of
$T_{\mu(A)}$ must be
$$(1,\sigma(1)),(2,\sigma(2)),\ldots,(n,\sigma(n)),$$
where $\sigma\in\mathfrak{S}_n$ and we now prove that we can choose them in such a way that they are the elements of a heap $H_n^*$. It remains to define the order $\le$ of the elements of the heap $H_n^*$.

The correspondence ${\sf s}(A)$ is defined as follows. Suppose the forest $T_{\mu(A)}$ contains exactly $k$ complete non-ambiguous trees $T_{\mu(B_1)},T_{\mu(B_2)},\ldots,T_{\mu(B_k)}$ that are associated to the vertex sets $\mu(B_1),\mu(B_2),\ldots,\mu(B_k)$, respectively, and the leaves of every $T_{\mu(B_i)}$ can be identified as a permutation $\bar{\sigma}_i$ which is $\sigma$ restricted to a subset of $[n]$. Furthermore, suppose the leftmost leaf of $T_{\mu(B_i)}$ is $(a_i,b_i)$ for $i\le k$ and $a_1,a_2,\ldots,a_k$ is strictly increasing. Then by applying the bijection ${\sf q}$ given in Lemma~\ref{L:bicm} on every $B_i$, we have the corresponding pyramid ${\sf q}(B_i)$ such that the unique maximal element is $(a_i,b_i)$ and all the elements of ${\sf q}(B_i)$ are exactly the paths of a permutation $\bar{\sigma}_i$. For every $i,j$ such that $1\le i\le k$ and $j>i$, and every element $x$ in the pyramid ${\sf q}(B_i)$, suppose $y_{x}$ is one of the minimal elements in the pyramid ${\sf q}(B_j)$ that is crossing with $x$ but not crossing with any $z$ such that $x\le z$ in the pyramid ${\sf q}(B_i)$. Then the heap
${\sf s}(A)$ is obtained by letting $y_{x}$ cover $x$ for every $x$ and $y_x$.

We next prove that the map $A\mapsto {\sf s}(A)$ is a bijection by showing its inverse map.
For a given heap $H_n^*$ in $\CMcal{H}(\CMcal{B}_3,\CMcal{R}_3)$ such that $H_n^*$ has $n$ elements which are exactly the paths of a permutation $\sigma$ in $\mathfrak{S}_n$, namely the elements of $H_n^*$ are $(1,\sigma(1)),\ldots,(n,\sigma(n))$, we consider the pyramid $P_1$ induced by the element $(1,\sigma(1))$, i.e., $P_1$ contains all the elements $y$ in $H_n^*$ such that $y\le (1,\sigma(1))$. From the remaining heap $H_n^*-P_1$, we choose an element $(c_1,d_1)$ such that $c_1$ is minimal of all the elements contained in $H_n^*-P_1$. Let $P_2$ be a pyramid induced by the element $(c_1,d_1)$, i.e., $P_2$ contains all the elements $y$ in $H_n^*-P_1$ such that $y\le (c_1,d_1)$. We continue this process until no element is left and we decompose $H_n^*$ into a sequence of pyramids $P_1,P_2,\ldots,P_k$ whose unique maximal elements are respectively $(1,\sigma(1)),(c_1,d_1),\ldots, (c_{k-1},d_{k-1})$ such that $1,c_1,\ldots,c_{k-1}$ is strictly increasing. In fact, $P_i={\sf q}(B_i)$. For any two crossing paths $(i,m_i),(j,m_j)$ such that $(i,m_i)\in P_s$ and $(j,m_j)\in P_t$, we have $(i,m_i)\le (j,m_j)$ if $s<t$. By applying the inverse bijection ${\sf q}^{-1}$ given in
Lemma~\ref{L:bicm} on the pyramids $P_1,P_2,\ldots,P_k$, we retrieve a
forest of binary trees $T_{\mu(B_1)},T_{\mu(B_2)},\ldots,T_{\mu(B_k)}$ whose vertex
set is a complete non-ambiguous forest, from which it follows that
$A\mapsto {\sf s}(A)$ is a bijection.

See Figure~\ref{F:8} for an example of the bijection ${\sf s}$.
\begin{figure}[htbp]
\begin{center}
\includegraphics[scale=0.8]{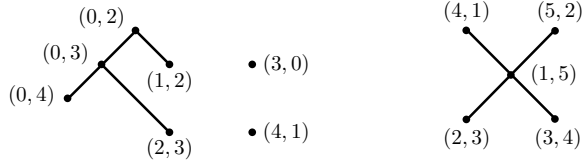}
\caption{A complete non-ambiguous forest with $5$ leaves (left) and its corresponding heap $H_5^*$ whose elements are exactly the paths of the permutation $\sigma=53412$ (right). Furthermore, the heap $H_5^*$ uniquely corresponds with the pair $(54132,21543)$ of permutations with no common rise.
\label{F:8}}
\end{center}
\end{figure}

In the next step, we will show that $h_{n,2}=\omega(n)$ by constructing a
bijection between the set of pairs $(\pi,\xi)$ of permutations of $[n]$ with no common rise and the set of heaps $H_n^*$ in $\CMcal{H}(\CMcal{B}_3,\CMcal{R}_3)$ such that $H_n^*$ has $n$ elements which are exactly the paths of a permutation $\sigma\in\mathfrak{S}_n$.

For a given pair $(\pi,\xi)$ of
permutations $\pi,\xi\in\mathfrak{S}_n$ without common rise, suppose
$\pi=a_1a_2\cdots a_n$ and $\xi=b_1b_2\cdots b_n$. Then, we choose a
permutation $\sigma\in\mathfrak{S}_n$ satisfying $\sigma(a_i)=b_i$
for every $i$. Let $(a_1,b_1),\ldots,(a_n,b_n)$ be the elements of the corresponding heap $H_n^*$, then it remains to define the order $\le $ in the heap $H_n^*$.
For every two crossing paths
$(a_i,b_i),(a_j,b_j)$, we set $(a_j,b_j)<(a_i,b_i)$ if $i<j$. Thus, we have given the map $(\pi,\xi)\mapsto H_n^*$.

We next prove this map is a bijection by showing its inverse map. Given a heap
$H_n^*$ in $\CMcal{H}(\CMcal{B}_3,\CMcal{R}_3)$ such that $H_n^*$ has $n$ elements which are exactly the paths of a permutation $\sigma\in \mathfrak{S}_n$, for any two paths
$(a_i,b_i),(a_j,b_j)$, we define $(a_i,b_i)\prec(a_j,b_j)$ if $(a_j,b_j)<(a_i,b_i)$ in the heap $H_n^*$. Otherwise, we define $(a_j,b_j)\prec(a_i,b_i)$ if $a_i<a_j$ and $(a_i,b_i)$, $(a_j,b_j)$ are incomparable in the heap $H_n^*$. In this way, we give a total order $\prec$ of the paths $(a_i,b_i)$ for all $i$ which allows us to rewrite the elements $(a_1,b_1),(a_2,b_2),\ldots,(a_n,b_n)$ by this linear order $\prec$. That is, suppose the sequence
$(a_{i_1},b_{i_1}),(a_{i_2},b_{i_2}),\ldots,(a_{i_n},b_{i_n})$ satisfies
$(a_{i_1},b_{i_1})\prec(a_{i_2},b_{i_2})\prec\cdots\prec(a_{i_n},b_{i_n})$, then
we choose $\pi=a_{i_1}a_{i_2}\cdots a_{i_n}\in\mathfrak{S}_n$ and
$\xi=b_{i_1}b_{i_2}\cdots b_{i_n}\in\mathfrak{S}_n$. There is no
common rise of the pair $(\pi,\xi)$. If there exists $i_j$ such
that $a_{i_j}<a_{i_{j+1}}$ and $b_{i_j}<b_{i_{j+1}}$, then the paths
$(a_{i_j},b_{i_j})$ and $(a_{i_{j+1}},b_{i_{j+1}})$ are not crossing, so they have no
cover relation in the heap $H_n^*$. It follows that $(a_{i_j},b_{i_j})$ and $(a_{i_{j+1}},b_{i_{j+1}})$ are incomparable in the heap $H_n^*$ and, in view of $a_{i_j}<a_{i_{j+1}}$, we have $(a_{i_{j+1}},b_{i_{j+1}})\prec (a_{i_{j}},b_{i_{j}})$ which contradicts the assumption that $(a_{i_j},b_{i_j})\prec(a_{i_{j+1}},b_{i_{j+1}})$.
Consequently, the pair $(\pi,\xi)$ has no common rise and the map $(\pi,\xi)\mapsto H_n^*$ is a bijection.
\end{proof}
See Figure~\ref{F:8} for an example of this bijection.

\section*{Acknowledgement}
The author thanks two anonymous reviewers for their very helpful suggestions and comments. The author also thanks Jiang Zeng, Zhicong Lin, Ren\'{e} Ciak and Laura Silverstein for related discussions and would like to give special thanks to the joint seminar Arbeitsgemeinschaft Diskrete Mathematik for their valuable feedback. This work was done during my stay in the AG Algorithm and Complexity, Technische Universit\"{a}t Kaiserslautern, Germany. The author would like to thank Markus Nebel, Sebastian Wild and Raphael Reitzig for their kind help and support.

\end{document}